\documentclass[12pt]{amsart}
\usepackage[a4paper,left=2cm,right=1.2cm,top=3cm,bottom=4cm]{geometry}
\usepackage{amsmath,amstext, amsthm,
amsbsy,amssymb,marvosym,fancyhdr,graphicx,amscd,amsfonts,latexsym,delarray,stackrel,lineno,color,cite,appendix,xcolor,url,hyperref}
\usepackage{wasysym}
\usepackage{subcaption}

\usepackage{srcltx}
\usepackage{mathrsfs}
\usepackage{float}

\usepackage[all]{xy}
\usepackage{t1enc}
\usepackage{mathrsfs}
\usepackage{pifont}

\usepackage{mathpple}

\usepackage[T1]{fontenc}

\definecolor{Green}{rgb}{0,1,0}
\definecolor{Blue}{RGB}{0,0,191}
\definecolor{mathmodecolor}{RGB}{0,102,0}
\definecolor{keywordcolor}{RGB}{0,51,151}
\definecolor{sourcebackgroundcolor}{RGB}{255,247,223}
\definecolor{unixagred}{RGB}{255,0,0}
\definecolor{lightgray}{RGB}{191,191,191}
\definecolor{green}{RGB}{1,191,191}

\newcommand*\patchAmsMathEnvironmentForLineno[1]{%
  \expandafter\let\csname old#1\expandafter\endcsname\csname #1\endcsname
  \expandafter\let\csname oldend#1\expandafter\endcsname\csname end#1\endcsname
  \renewenvironment{#1}%
     {\linenomath\csname old#1\endcsname}%
     {\csname oldend#1\endcsname\endlinenomath}}%
\newcommand*\patchBothAmsMathEnvironmentsForLineno[1]{%
  \patchAmsMathEnvironmentForLineno{#1}%
  \patchAmsMathEnvironmentForLineno{#1*}}%
\AtBeginDocument{%
\patchBothAmsMathEnvironmentsForLineno{equation}%
\patchBothAmsMathEnvironmentsForLineno{align}%
\patchBothAmsMathEnvironmentsForLineno{flalign}%
\patchBothAmsMathEnvironmentsForLineno{alignat}%
\patchBothAmsMathEnvironmentsForLineno{gather}%
\patchBothAmsMathEnvironmentsForLineno{multline}%
}

\newtheorem{thm}{Theorem}[section]

\newtheorem{prop}[thm]{Proposition}
\newtheorem{corl}[thm]{Corollary}
\newtheorem{lma}[thm]{Lemma}

\newtheorem{defn}[thm]{Definition}
\newtheorem{rem}[thm]{Remark}
\newtheorem{example}[thm]{Example}

\def\coker{{\rm coker}}

\def\env{{\rm env}}

\def\Gr{{\rm Gr}}

\def\id{{\rm id}}

\DeclareMathOperator{\rk}{rk}

\def\S{\mathbb {S}}

\newcommand{\cT}[1]{C(S^1)^{(#1)}}
\newcommand{\FR}[1]{C(S^1)_{(#1)}}
\newcommand{\cC}[1]{\mathcal{C}^{(#1)}}

\def\tilde{\widetilde}

\def\tr{{\rm tr}}
\def\V{\mathcal{V}}

\def\tr{{\rm tr}}

\def\cE{{\mathcal E}}

\def\cK{{\mathcal K}}

\def\cN{{\mathcal N}}
\def\cR{{\mathcal R}}

\def\M{{\mathcal M}}

\def\D{\mathbb{D}}
\def\N{\mathbb{N}}
\def\R{\mathbb{R}}
\def\P{\mathbb{P}}
\def\I{\mathbb{I}}
\def\Z{\mathbb{Z}}
\def\C{\mathbb{C}}

\def\dar[#1]{\ar@<2pt>[#1]\ar@<-2pt>[#1]}

\newcommand{\nil}[1]{}

\makeatletter
\DeclareMathOperator{\exterior}{\@ifnextchar^\@exterior{\@exterior^{}}}
\def\@exterior^#1{\mathop{\bigwedge\nolimits^{\!#1}}}
\makeatother

\begin{document}

      \title[K-theoretic invariants of the Toeplitz operator system, its dual; graph operator systems]{K-theoretic invariants of the Toeplitz operator system and its dual, and of all graph operator systems}
\author{Walter D. van Suijlekom}

\email{waltervs@math.ru.nl}
\address{Institute for Mathematics, Astrophysics and Particle Physics, Radboud University Nijmegen, Heyendaalseweg 135, 6525 AJ Nijmegen, The Netherlands.}

\date{30 September 2026}

\begin{abstract}
  We compute $K$-theoretic invariants for several finite-dimensional operator systems. These include the Toeplitz operator system and its dual (the Fej\'er--Riesz operator system), as well as operator systems associated to tolerance relations, aka graph operator systems. A variety of techniques is applied, adapted to each of these cases, ranging from a generalization of Carath\'eodory's factorization result for block Toeplitz matrices, to Wiener--Hopf factorization and Schur complements. In most instances we find that the invariants coincide with their analogues for the $C^*$-envelope, except for the $K_1$-invariants of the Fej\'er--Riesz operator system.
\end{abstract}

\maketitle

\setcounter{tocdepth}{2}
\tableofcontents

\section{Introduction}
In our search for an extension of K-theory from $C^*$-algebras to operator systems  \cite{Sui24,Sui25b} we found that one can indeed formulate such a theory which is based on hermitian forms. This was inspired by the Witt group for rings \cite{Knu91,Bal05} which for unital $C^*$-algebras indeed coincides with the $K_0$-group \cite{Ros95}. Moreover, almost-projections (in the sense of \cite{OY15}) and projections in operator systems \cite{AR22} were found to be compatible with the proposed structures.

We showed in \cite{Sui24} that the resulting theory generalizes $C^*$-algebraic $K$-theory, that it is compatible with ucp maps that lift to the $C^*$-envelope, that it extends to non-unital operator systems, and that it is invariant under Morita equivalence for operator systems (in the sense of \cite{EKT21}).  

However, besides this development of the general theory, the question also arose whether it would be possible to compute these $K$-theoretic invariants for some examples. Indeed, one of the powerful aspects of $C^*$-algebraic $K$-theory is that in many cases of interest, these invariants can be explicitly computed. Recently, some explicit examples of operator systems were encountered in the context of noncommutative geometry: the Toeplitz operator system $\cT{n}$ arising from spectral truncations of the circle \cite{CS20}, its dual given by the Fej\'er--Riesz operator system $\FR{n}$ \cite{CS20,Far21}, and tolerance relations on metric spaces \cite{CS21,GS22}. Concretely, the first consists of $n \times n $ Toeplitz matrices, the second of Fourier truncations on the circle, while the third can be realized by graph operator systems (in the finite-dimensional case). As we will see below, the computation of $K$-theoretic invariants in all of these cases can indeed be done, and comes down to analyzing the number of path-components in the space of (hermitian) invertible matrices with values in the operator system, for which in each of these cases a wealth of literature is available; we will indicate these below in each section. Note that a brief discussion on the $K$-invariants for four three-dimensional operator systems is given in \cite{Sui26a}.

In Section \ref{sect:Toep} we analyze path-connected components in the space of hermitian forms in the Toeplitz operator system $\cT{n}$. We find that the embedding in its $C^*$-envelope $M_n(\C)$ actually induces an isomorphism between the $K_0$-invariants $\V_0(\cT{n},m)$ and $\V_0(M_n(\C),m)$. In particular, $K_0(\cT{n}) \cong K_0(M_n(\C)) \cong \Z$.

In Section \ref{sect:FR} we continue with the dual of the Toeplitz operator system: the so-called Fej\'er--Riesz operator system $\FR{n}$. Concretely, it is given by functions $a(z) = \sum a_j z^j$ with fixed support in Fourier, {\em i.e.} for which $a_j =0$ for all $|j| \geq n$. Upon analyzing the path-components in the space of hermitian forms, we again find that the embedding of it in its $C^*$-envelope $C(S^1)$ induces isomorphisms between $\V_0(\FR{n},m)$ and $\V_0(C(S^1),m)$. Consequently, for the $K_0$-groups we have $K_0(\FR{n}) \cong K_0 (C(S^1)) \cong \Z$.

We also compute the invariant $\V_1(\FR{n},m)$. For finite-dimensional $C^*$-envelopes, this invariant is known to be trivial (see Proposition \ref{prop:K1} below), but in the case of the Fej\'er--Riesz operator system we find that it gives non-trivial information. Indeed, we find that the determinant winding number gives an isomorphism $\V_1(\FR{n},m) \cong \{ -m(n-1), -m(n-1)+1, \ldots, m(n-1) \}$, in contrast to $\V_1(C(S^1),m) \cong \Z$.

In Section \ref{sect:tol} we consider the operator system $E(\cR)$ associated to a tolerance (= reflexive and symmetric) relation $\cR$ on a finite set $X= \{ 1,\ldots, n\}$. Explicitly, this operator system consists of $n \times n$ matrices whose entries at $(i,j)$ are zero whenever $(i,j) \notin \cR$. They were also called graph operator systems \cite{OP15}: the elements of $\cR$ can be considered as edges on the elements (vertices) in $X$. The $C^*$-envelope is given by a direct sum of matrix algebras $A(\cR) := M_{n_1}(\C) \oplus \cdots \oplus M_{n_c}(\C)$, where the summands correspond to the connected components (on $n_1, \ldots, n_c$ vertices) of the underlying graph. Once again, we find that the embedding in its $C^*$-envelope induces isomorphisms between $\V_0(E(\cR),m)$ and $\V_0( A(\cR) ,m)$. Consequently, for the $K_0$-groups we have $K_0(E(\cR)) \cong K_0 (A(\cR) ) \cong \Z^c$.

\subsection*{Acknowledgments}
We thank Kristin Courtney, Maximilian Illmer, Sam Kim, Yuezhao Li, Tim Netzer, Lukas Obermeyer and Hermann Schulz--Baldes for fruitful discussions concerning the derivation of the $K$-invariants for the operator systems presented here.

For this manuscript, we have made use of LLMs. Specifically, there have been dialogues of the author with ChatGPT-5.6 Sol and 6 Astra (OpenAI), Gemini 3.1 Pro (Google), Claude Opus 5 (Anthropic), in the search for a derivation of the number of path-components in the space of (hermitian) invertibles in the operator systems under consideration. For the Toeplitz operator system, the general scheme was set up by the author, taking Carath\'eodory factorization of invertible hermitian block Toeplitz matrices as a starting point. The subsequent steps in Section \ref{sect:Toep} were then a result of a continuous going back-and-forth to the LLM-prompts to strengthen, clarify or rephrase the argument. 

The skeletons for the derivations of the $K_0$-invariants of the Fej\'er-Riesz operator system and the graph operator systems were produced by Claude Opus 5 and for $K_1$ of the Fej\'er--Riesz operator system by ChatGPT-6 Astra. This was then further developed, rephrased and improved by the author, including the addition of proper citations, to arrive at the texts in Sections \ref{sect:FR} and \ref{sect:tol}. Open weight LLMs (QWen3.8) were used for proofreading. That being said, the responsibility for the content below lies entirely with the author.

\section{General theory}
We recall from \cite{Sui24,Sui25b} the notion of non-singular elements and hermitian forms in an operator system. For our purposes it is sufficient to consider {\em concrete} operator systems, which are defined to be $*$-closed and norm-closed vector spaces of bounded operators on a Hilbert space and containing the unit 1. In fact, we realize an operator system $E$ inside its $C^*$-envelope via the map $\imath_E : E \to C^*_\env(E)$. For more details, we refer {\em e.g.} to \cite{ER00,Pau02,Pis03}

\begin{defn}
\label{defn:non-sing}
Let $(E,e)$ be a unital operator system. An element $x \in M_m(E)$is called {\em non-singular} if $\imath_E^{(m)}(x)$ is invertible as a matrix with entries in the $C^*$-envelope $C^*_\env(E)$. 

A {\em hermitian form} is a non-singular element $x \in M_m(E)$ which is self-adjoint. 
\end{defn}
We will write $H(E,m)$ for all hermitian forms in $M_m(E)$ and $G(E,m)$ for all non-singular elements in $M_m(E)$. Also, set $\V_0(E,m) := \pi_0(H(E,m))$ and $\V_1(E,m) := \pi_0(G(E,m))$, {\em i.e.}, the sets of path-components of hermitian forms and non-singular elements in $M_m(E)$, respectively.  
We write $x \sim_m x'$ if $x,x'$ belong to the same path component in either $H(E,m)$ or $G(E,m)$.
All the sets $\V_p(E,m)$ ($p=0,1$) are invariants of operator systems in the following sense. 
\begin{prop}
  If $E$ and $F$ are completely order isomorphic then $\V_p(E,m) \cong \V_p(F,m)$ for all $m \geq 1$.
\end{prop}

\begin{example}
  \label{ex:V-C}
  Consider $E=\C$. Then $\V_0(\C,m)$ is the set of path components of invertible hermitian $m \times m$ matrices. Any such matrix $x$ can be diagonalized with a unitary matrix, and since the unitary group $U(m)$ is connected, there is a continuous path between $x$ and the corresponding diagonal matrix. In turn, this diagonal matrix is path-connected (in the space of invertible hermitian matrices) to the corresponding signature matrix, which is unique up to ordering.  In other words, $\V_0(\C,m)$ can be parametrized by the signature $s$ of 
  the hermitian forms
  , yielding an isomorphism
  \begin{equation}
    \label{eq:V-C-n}
  \V_0(\C,m) \cong \{  -m , -m+2 ,\ldots, m-2, m \}.
 \end{equation}
\end{example}
Concerning the invariant $\V_1$, we record the following from \cite{Sui25b}:
 \begin{prop}
      \label{prop:K1}
Suppose that $E$ is a unital operator system with finite-dimensional $C^*$-envelope. Then for any $m \geq 1$ the space $G(E,m)$ is path connected. Consequently, $\V_1(E,m) = \{ [1^{\oplus m}] \}$ for all $m \geq 1$.
  \end{prop}

 Of interest is also the limit structure of increasing matrix size. For this we consider the direct system of sets $(\V_p(E,m), \imath_{mm'})$ where for $m' \geq m$
\begin{align}
  \label{eq:dir-syst}
  \imath_{mm'}: \V_p(E,m) &\to \V_p(E,m');\\
[x]_m & \mapsto [x \oplus e_{m'-m}]_{m'}. \nonumber
\end{align}
The corresponding direct limit structure can also be described more concretely by the following equivalence relation: for $x \in G(E,m)$ and $x' \in G(E,m')$ we write $x \sim x'$ if there exists a $k \geq m,m'$ such that $ x \oplus e_{k-m} \sim_k x'\oplus e_{k-m'} $ in $G(E,k)$ (and similarly for hermitian forms in $H(E,m)$). We will write $[x]_E$, or simply $[x]$, for the equivalence class corresponding to $x \in G(E,m)$ or $H(E,m)$ and $\V_p(E) := \amalg_m \V_p(E,m)/_\sim$ for the corresponding set of equivalence classes. 
\begin{prop}
  The set $\V_p(E)$ is the direct limit $\varinjlim \V_p(E,m)$ of the direct system \eqref{eq:dir-syst}. Moreover, it is a semigroup when equipped with the direct sum $[x]+ [x'] = [x \oplus x']$ and identity element $0 = [e]$. 
\end{prop}

\begin{defn}
  Let $(E,e)$ be a unital operator system and let $p=0,1$. We define the $K$-theory group $K_p(E)$ of $E$ to be the Grothendieck group of $\V_p(E)$. 
\end{defn}

\begin{example}
  \label{ex:K-C}
  Continuing Example \ref{ex:V-C}, consider $E=\C$. Note that the direct sum of two hermitian forms translates to the addition of the corresponding signatures: $(s,s')\mapsto s+s'$.
  
The connecting maps can be expressed in terms of the signature as $\imath_{mm'} (s) = s + m'-m$. Moreover, there are commuting diagrams for any $m' \geq m \geq 1$:
 \begin{equation*}
     \xymatrix {
       \V_0(\C,m) \ar[rd]_{\rho_m} \ar[rr]^{\imath_{mm'}}&& \V_0(\C,m') \ar[ld]^{\rho_{m'}}\\
&       \N 
     }
  \end{equation*}
 where $\rho_m([x])$ is defined to be the so-called {\em negative index of inertia} of the representative $x$, {\em i.e.} the number of negative eigenvalues of the hermitian form $x$. When expressed in terms of the signature $s$ of $x$, we have $\rho_m(s) =\frac 12 (m-s)$ from which commutativity of the diagram follows at once. This yields that $\varinjlim \V_0(\C,m) \cong \N$, compatibly with $K_0(\C) \cong \Z$, as it should. 
\end{example}

We end this section with three crucial properties of the above $K_p$-groups, obtained in \cite{Sui24,Sui25b}. 
\begin{prop}
\label{prop:K0-properties}
  \begin{enumerate}
  \item   For a unital $C^*$-algebra $A$ the group $K_p(A)$ is isomorphic to the $C^*$-algebraic $K_p$-theory group of $A$.
    \item For two unital operator systems $E$ and $F$ we have $\V_p(E \oplus F ,m) \cong \V_p(E,m) \times \V_p(F,m)$.
\item   Let $E$ be a unital operator system and let $N$ be a natural number. Then $\V_p(E)$ is isomorphic to $\V_p(M_N(E))$ (and so are the corresponding $K_p$-groups).
  \end{enumerate}
    \end{prop}

\section{$K$-invariants of the Toeplitz operator system}
\label{sect:Toep}
We write $\cT{n}$ for the operator system of $n \times n$ Toeplitz matrices, {\em i.e.} matrices of the form
  \begin{equation}
  T = \begin{pmatrix}
t_0 & t_{-1} & t_{-2} & \cdots & t_{-n+1} \\
t_1 & t_0 & t_{-1} & \cdots & t_{-n+2} \\
t_2 & t_1 & t_0 & \cdots & t_{-n+3} \\
\vdots & \vdots & \vdots & \ddots & \vdots \\
t_{n-1} & t_{n-2} & t_{n-3} & \cdots & t_0
  \end{pmatrix}.
  \label{eq:toeplitz}
  \end{equation}
  This operator system was considered at length in \cite{CS20}, arising from spectral truncations of the circle, and where it was also shown that the $C^*$-envelope of this operator system is $M_n(\C)$.
  
  For the $K_0$-invariants $\V_0(\cT{n},m)$ we need to consider hermitian $m \times m$ matrices with entries in the Toeplitz matrices. By identifying $M_m(\cT{n})\cong \cT{n} \otimes M_m(\C)$, we may equivalently consider block Toeplitz matrices of the following form:
  \begin{equation}
  T = \begin{pmatrix}
T_0 & T_{-1} & T_{-2} & \cdots & T_{-n+1} \\
T_1 & T_0 & T_{-1} & \cdots & T_{-n+2} \\
T_2 & T_1 & T_0 & \cdots & T_{-n+3} \\
\vdots & \vdots & \vdots & \ddots & \vdots \\
T_{n-1} & T_{n-2} & T_{n-3} & \cdots & T_0
  \end{pmatrix}; \qquad (T_j \in M_m(\C)).
  \label{eq:toeplitz-block}
  \end{equation}
  We note that $T \in M_m(\cT{n})$ is a hermitian form iff it is hermitian and non-singular as a block Toeplitz matrix.

\subsection{Path components of non-singular hermitian block Toeplitz matrices}
  Let us now analyze the $K_0$-invariants $\V_0(\cT{n},m)$ introduced before for the case of the Toeplitz operator system. This amounts to analyzing the number of path-components of the non-singular hermitian block Toeplitz matrices. 

\begin{thm}
  \label{thm:toepl}
  \begin{enumerate}
    \item 
      The space of non-singular hermitian block Toeplitz matrices with signature $(p,q)$, denoted by $H(\cT{n},m)_{p,q}$, is path-connected.
    \item The space $H(\cT{n},m)$ of non-singular hermitian block Toeplitz matrices has $nm+1$ path-connected components, which are labeled by matrix signatures. Consequently, $\V_0(\cT{n},m) \cong \{-nm, -nm+2, \ldots, nm\}$.
      \end{enumerate}
\end{thm}
Another way of phrasing this is that the inclusion map $\imath_\env^{\cT{n}} : \cT{n} \to C^*_\env(\cT{n}) \cong M_n(\C)$ induces isomorphisms $\V_0(\cT{n},m) \cong \V_0( M_n(\C), m)$ for any $m \geq 1$, intertwining the connecting maps $\imath_{m m'}$ from Equation \eqref{eq:dir-syst} for $E= \cT{n}$ with those of $E = M_n(\C)$. Combining this with Proposition \ref{prop:K0-properties}(3) it follows that  $K_0(\cT{n}) \cong K_0 (C^*_\env (\cT{n})) \cong \Z$.

\bigskip

Our proof of Theorem \ref{thm:toepl} is based on the parametrization of non-singular hermitian block Toeplitz matrices by Krylov matrices \cite{FH88}.
For any $A \in M_{nm}(\C)$ and $X \in M_{nm,m}(\C)$ we define an $nm\times nm$ {\em Krylov matrix} by
$$
K(X,A) = \begin{pmatrix} X & AX & \cdots & A^{n-1} X \end{pmatrix}.
$$ 
Also, we define for signatures $(p,q)$ ({\em i.e.} pairs of non-negative integers such that $p+q = nm$) the space
$$
\cC{n,m}_{p,q} = \{ (X,W) \mid X \in M_{nm,m}(\C), W \in U(p,q) : \det K(X,W) \neq 0 \},
$$
where $U(p,q)$ is the indefinite unitary group in $M_{nm}(\C)$:
$$
U(p,q):= \{ W \in M_{nm}(\C) \mid W^* \Lambda W = \Lambda \}, \qquad \Lambda = \I_p\oplus - \I_q .
$$
Freund and Huckle then established the following generalization of Carath\'eodory's factorization result to non-singular hermitian block Toeplitz matrices \cite{FH88}.
\begin{prop}[{\cite[Corollary 1]{FH88}}]
  \label{prop:param-Toepl-block}
  The following map is surjective and continuous:
  \begin{align*}
    \pi^{(m)} : \cC{n,m}_{p,q} &\to H(\cT{n},m)_{p,q};\\
    (X,W) & \mapsto K(X,W)^* \cdot \Lambda \cdot K(X,W).
    \end{align*}
\end{prop}
\proof
First note that $\pi^{(m)} (X,W)$ is indeed a hermitian block Toeplitz matrix since its $(i,j)$'th entry is
\begin{equation}
\left( K(X,W)^* \cdot \Lambda \cdot K(X,W) \right)_{ij} = \left\{ \begin{array}{ll} X^* (W^*)^{i-j} \Lambda  X & \text{if } i \geq j \\
  X^* \Lambda W^{j-i} X & \text{if } j \geq i  \end{array} \right.
\label{eq:pi-toepl}
\end{equation}
and which, by Sylvester's Law (for invertible $K(X,W)$), indeed has signature $(p,q)$.

Suppose $T$ is a hermitian invertible block Toeplitz matrix. For $j=1,\ldots n$ we write
$$
E_j : \C^m \to \C^{nm}
$$
for the canonical embedding, so that $\begin{pmatrix} E_1 & E_2 & \cdots & E_{n} \end{pmatrix} = \I_{nm}$. The vector space $\C^{nm}$ has a non-degenerate hermitian form, given by $T$ as:
$$
[x,y]_T := x^* T y ; \qquad (x,y \in\C^{nm}). 
$$
We introduce the subspaces
$$
\cE:= \text{Ran} \{ E_1, \ldots, E_{n-1} \}; \qquad
\mathcal F:= \text{Ran} \{ E_2, \ldots, E_{n} \}
$$
and a block shift operator
$$
S: \cE \to \mathcal F ; \qquad S E_j = E_{j+1}. 
$$
One easily checks that for any $v,w \in \C^m$ we have 
$$
[SE_i v, S E_j w ]_T = v^* E_{i+1}^* T^* E_{j+1} w = v^* T_{i-j} w = [ E_i v, E_j w ]_T. 
$$
In other words, $S: \cE \to \mathcal F$ is an isometry on subspaces in $(\C^{nm}, [\cdot, \cdot]_T)$. By Witt's extension theorem there exists an extension $U: \C^{nm} \to\C^{nm}$ such that $U^* T U = T$.

By Sylvester's Law, for $T$ hermitian invertible and, say, with signature $(p,q)$, there exists $V \in GL_{nm}(\C)$ such that
$$
T = V^* \Lambda V; \qquad \Lambda = \I_p \oplus - \I_q. 
$$
Set $X = VE_1 \in M_{nm,m}(\C)$ and $W = VUV^{-1}$. Then
$$
W^* \Lambda W = (V^{-1})^* U^* V^* \Lambda V U V^{-1} =  (V^{-1})^* U^* T U V^{-1} =  (V^{-1})^* T  V^{-1} =  \Lambda.
$$
So $W \in U(p,q)$. Moreover,
$$
W^k X = VU^k E_1 = VE_{k+1}; \qquad (k=0,\ldots, n-1).
$$
As such,
$$
K(X,W) = V \begin{pmatrix} E_1 & E_2 & \cdots & E_n \end{pmatrix} = V
$$
and this implies that
\[
T = V^* \Lambda V = K(X,W)^* \Lambda K(X,W). \qedhere
\]
\endproof

As a preparation for the study of $\cC{n,m}_{p,q}$, we start with some preliminary notions and results.
\begin{defn}
  Let $n,m \geq 1$. A matrix $A \in M_{nm}(\C)$ is called {\em $m$-cyclic} if there exists an $X \in M_{nm,m}(\C)$ such that $\det K(X,A) \neq 0$.
\end{defn}
  We also introduce the (finite and infinite) cyclic subspace generated by $A$ and a vector $y \in \C^{nm}$:
\begin{align*}
Z_k (y,A) &:= \text{span}_\C \{ y, Ay, \ldots, A^{k-1}y\},\\
Z(y,A)&:=  \text{span}_\C \{ y, Ay, A^2 y \ldots\}.
\end{align*}
We now consider for each $r=1,\ldots ,n$ the spaces of all vectors $y \neq 0$ for which $Z_n(y,A)$ has dimension $r$:
\begin{align*}
  Y_r(A) := \left \{ y \in \C^{nm}\mid \dim_\C Z_n(y,A) = r  \right\}.
\end{align*}
We may decompose $y$ according to the generalized eigenspaces (or Jordan decomposition) of $A$ and write $y = \sum_\lambda y_\lambda$. We write the {\em height} of $y_\lambda$ as
$$
h_\lambda(y) = \min \{ h \geq 0: (A-\lambda \I)^{h} y_\lambda = 0 \}.
$$
\begin{lma}
For any $r=1,\ldots, n-1$ it holds that $y \in Y_r(A)$ iff $\sum_\lambda h_\lambda(y) = r$.
  \end{lma}
\proof 
The above decomposition $y = \sum_\lambda y_\lambda$ with $y_\lambda$ of height $h_\lambda(y)$ allows to write the unique monic polynomial $m_y(t)$ of least degree such that $m_y(A) y = 0$ as the following product:
$$
m_y (t) = \prod_\lambda (t-\lambda )^{h_\lambda}.
$$
The evaluation map $\C[t] \to Z(y,A), p \mapsto p(A) y$ has kernel $(m_y)$ from which we conclude that $Z(y,A)$ has dimension $\deg m_y$. In combination with the previous paragraph, we may thus conclude that $\dim Z(y,A)= \sum_\lambda h_\lambda$. In other words, $\dim Z(y,A) = r$ iff $\sum_\lambda h_\lambda = r$. 
Now, suppose that $d$ is the smallest integer for which $A^d y$ is in the linear span of $y, Ay, \ldots, A^{d-1}y$. Then all higher powers $A^k y$ with $k > d$ are also in this linear span and $\dim Z(y,A) = d$. In fact, for all $d' > d$ we already have $\dim Z_{d'}(y,A) = d$. Hence, under the assumption that $r<n$ we find $\dim Z(y,A) = r$ iff $\dim Z_n(y,A) = r$. 
\endproof

For a general height vector $\{ h_\lambda\}$ we write
$$
Y_{\{ h_\lambda \}}(A) := \{ y \in \C^{nm} \mid h_\lambda(y) = h_\lambda \}.
$$
This is an open subset of the following complex linear space
$$
V_{\{ h_\lambda\}}= \oplus_\lambda \ker (A-\lambda \I)^{h_\lambda}.
$$
\begin{lma}
  \label{lma:aux-dimker}

  \begin{enumerate}
    \item For any $r=1,\ldots, n-1$ we have
  $$
Y_r(A) = \bigsqcup_{\sum h_\lambda =r } Y_{\{h_\lambda\}} (A)
$$
Moreover, if $\dim ker (A-\lambda \I) \leq m$ for all $\lambda \in \C$ and $\{ h_\lambda\}$ is a height vector such that $\sum_\lambda  h_\lambda =r$, then $Y_{\{ h_\lambda\}} (A)$ is a complex manifold of dimension $\leq m r$.
\item For $r=n$ we have that $Y_n(A)$ is itself a complex manifold of dimension $nm$.
  \end{enumerate}
\end{lma}
\proof
(1) Note that from the assumption $\dim \ker (A-\lambda \I) \leq m$ it follows that $\dim \ker (A-\lambda \I)^{h_\lambda} \leq m h_\lambda$ for all $\lambda$. Hence, for height vectors with the stated assumption we find
\[
\dim(Y_{\{h_\lambda\}}(A)) \leq  \dim V_{\{h_\lambda\}} \leq \sum_\lambda m h_\lambda = mr.
\]
(2) The space $Y_n(A)$ is characterized by the condition that $\rk \begin{pmatrix} y & Ay & \cdots & A^{n-1} y \end{pmatrix} = n$. Equivalently, at least one $n \times n$ minor is non-zero, hence it is an open subset of $\C^{nm}$. \endproof

\begin{prop}
  \label{prop:m-cyclic-dimker}
A matrix $A \in M_{nm}(\C)$ is $m$-cyclic if and only if $\dim ker (A-\lambda \I) \leq m$ for all $\lambda \in \C$.
\end{prop}
\proof
Suppose that $A$ is $m$-cyclic, so that $K(X,A)$ is non-singular for some $X \in M_{nm,m}(\C)$; fix such an $X$. Arguing by contradiction, assume that $\dim \ker(A - \lambda \I) > m$. Then also $\dim \ker (A^* -\bar \lambda \I) > m$ so that the map
$$
\ker (A^* -\bar \lambda \I) \to \C^m; \qquad y \mapsto X^* \cdot y
$$
must have a non-zero kernel. Take a non-zero vector $y$ in this kernel, so that $X^*\cdot y = 0$. Since also $A^*\cdot y = \bar \lambda y$ we find for any $k$:
$$
y^* \cdot A^k \cdot X = \lambda^k y^* \cdot X = 0. 
$$
This implies that $y^*\cdot K(X,A) = 0$, contradicting the assumed invertibility.

In the other direction, note that $K(X,A)$ is singular if and only if there exists a non-zero $y \in \C^{nm}$ such that $y \in \ker K(X,A)^*$. This is equivalent to $X^* (A^*)^k \cdot y = 0$ for all $k=0,\ldots, n-1$, which in turn amounts to
$$
\text{Ran} (X) \subseteq Z_n(y,A^*)^\perp.
$$
We will establish that if $\dim \ker (A-\lambda \I) \leq m$ then this condition fails for some $X$. Our strategy is to show that the space for which it holds decomposes into pieces each of which has (real) codimension $\geq 2$. Consequently, it is a proper subset, leaving an $X$ in its complement for which $K(X,A)$ is non-singular.

Recall the decomposition in Lemma \ref{lma:aux-dimker} of $Y_r(A)$ into complex manifolds: for $r<n$ these are labeled by heights ${\{ h_\lambda\}}$ for which $\sum_\lambda h_\lambda =r$ and for $r=n$ already $Y_n(A)$ itself is a complex manifold. Let us treat these two cases separately.

\begin{description}
\item[($r<n$)] For a height vector ${\{ h_\lambda\}}$ for which $\sum_\lambda h_\lambda =r$ we define:
$$
\Sigma_{\{h_\lambda\}}  = \left\{ ([y],X) \in \P(Y_{\{h_\lambda\}} (A^*)) \times M_{nm,m}(\C) \mid \text{Ran} (X) \subseteq Z_n(y,A^*)^\perp \right\}.
$$
We claim that the natural projection map $\rho: \Sigma_{\{h_\lambda\}} \to \P(Y_{\{h_\lambda\}} (A^*))$ turns $\Sigma_{\{h_\lambda\}}$ into a vector bundle over $\P(Y_{\{h_\lambda\}}(A^*))$. Indeed, the condition for $X$ to be in the fiber $\rho^{-1}([y])$ amounts to checking that the $m$ columns of $X$ are in $Z_n(y,A^*)^\perp$, and this varies smoothly in $y$. In other words,
$$
\rho^{-1}([y]) \cong (Z_n(y,A^*)^\perp)^{\oplus m}.
$$
Applying Lemma \ref{lma:aux-dimker}(1) to $A^*$ yields $\dim_\C \P(Y_{\{h_\lambda\}}(A^*)) \leq mr -1$ while the complex dimension of the fibers equals $m (nm- r)$ since $\dim Z_n(y,A^*) = r$. We thus conclude that $\Sigma_{\{h_\lambda\}}$ is a manifold with real $\dim_\R \leq 2 nm^2 -2$. 

\item[($r=n$)]  In this case, we define
$$
\Sigma_n = \left\{ ([y],X) \in \P(Y_n (A^*)) \times M_{nm,m}(\C) \mid \text{Ran} (X) \subseteq Z_n(y,A^*)^\perp \right\}.
$$
This is a rank $m(nm-n)$-vector bundle over the $nm-1$-dimensional complex base manifold $\P(Y_n (A^*))$. As such, $\Sigma_n$ is a manifold of real dimension $2 nm^2-2$. 
\end{description}
The set of matrices $X \in M_{nm,m}(\C)$ for which $K(X,A)$ is singular is precisely the projection of the union of all these manifolds $\Sigma_{\{h_\lambda\}}$ (for $r<n$) and $\Sigma_n$ onto the second coordinate. Given that the dimension of each of these is $\leq 2nm^2-2$, so that every point is critical for the projection onto $M_{nm,m}(\C)$. By Sard's Theorem this image has measure zero. Consequently, the (finite) union of these images still has measure zero, and hence cannot exhaust all of $M_{nm,m}(\C)$. As such, we can find an $X$ in the complement of their union, for which $K(X,A)$ is invertible. 
\endproof

\begin{rem}
  Note that there is an interesting link between the above approach to proving Proposition \ref{prop:m-cyclic-dimker} and control theory. We point the interested reader to {\em e.g.} \cite[Theorem 6.1]{Che99} or \cite[Theorem 5.2.27]{PW98} where it is shown that systems of the form 
  $$
\frac{d}{dt} x = Ax +Xu,
  $$
characterized by 
matrices $A \in M_N(\C), X \in M_{N,m}(\C)$, are controllable iff 
 the $N \times Nm$ Krylov matrix  $\begin{pmatrix} X & AX & \cdots A^{N-1}X \end{pmatrix}$ has rank $N$. In turn, this is equivalent to the condition that the $N \times (N+m)$ matrix $\begin{pmatrix}  A- \lambda \I & X\end{pmatrix}$ has rank $N$ for all $\lambda$. This is all very similar to our case of interest, {\em except} that we consider the Krylov matrix with powers up to $n-1$ and rephrase $m$-cyclicity as $\dim \ker (A-\lambda \I) \leq m$ instead of $N$. In terms of controllability indices, we are considering the special case where all of these are equal to $n$ ({\em cf.} \cite[Section 6.2.1]{Che99}). 
  \end{rem}

\begin{prop}
  \label{prop:path-conn-C-block}
The space $\cC{n,m}_{p,q}$ is path connected. 
  \end{prop}

We will approach this by analyzing the complement of $\cC{n,m}_{p,q} \subseteq M_{nm,m}(\C) \times U(p,q)$, for which we first prove some preliminary results. 
Let us write
\begin{align*}
\cC{n,m}_{p,q} &= M_{nm,m}(\C) \times U(p,q) \setminus \left(  Z_{p,q}^{nc} \cup  Z_{p,q}^{c} \right)
\intertext{where}
 Z_{p,q}^{nc} &= \{ (X,W)\in M_{nm,m}(\C) \times U(p,q)\mid W \text{ is not $m$-cyclic} \},\\
 Z_{p,q}^{c} &= \{ (X,W) \in M_{nm,m}(\C) \times U(p,q) \mid W \text{ is $m$-cyclic but } \det K(X,W) =0 \}.
\end{align*}
In what follows, we will first show that $ Z_{p,q}^{nc}$ decomposes into smooth manifolds of real codimension $\geq 2$, so that their complement in the path-connected real manifold $M_{nm,m}(\C) \times U(p,q)$ remains path-connected. A direct argument then applies to show that the subsequent complement of $Z_{p,q}^{c}$ remains path-connected.

For convenience we write $N=nm$ and $l=m+1$ (assuming $n\geq 2$, the $n=1$ case being trivial) and introduce the following sets
\begin{align*}
 \cN &:= \left\{ A \in M_{N}(\C) \text{ not $m$-cyclic} \right \},\\
 \cE &:= \left\{ (A,E) \in M_{N}(\C) \times \Gr(l,N) \mid A_{|E} = \lambda \cdot \id_E \text{ for some } \lambda \in \C \right\}.
\end{align*}
In view of Proposition \ref{prop:m-cyclic-dimker} the projection $\pi_1$ onto the first coordinate maps $\cE$ onto $\cN$. So let us start by analyzing the structure of $\cE$.

\begin{lma}
  \label{lma:dim-R}
  Let $\rho$ be the restriction to $\cE$ of the projection $\pi_2 : M_N(
  \C) \times \Gr(l,N) \to \Gr(l,N)$ onto the second coordinate. Then $\rho: \cE \to \Gr{(l,N)}$ is a holomorphic vector bundle of rank $N(N-l)+1$. Moreover, as a complex submanifold of $M_N(\C) \times Gr(l,N)$ it is closed and has dimension $N^2-l^2+1$. 
  \end{lma}
\proof
Take an affine chart $U$ of the Grassmannian consisting of the $l$-planes of the form
$$
E_B := \left\{ \begin{pmatrix} v \\ B v \end{pmatrix} \mid v \in \C^l \right\}
$$
for varying $B \in M_{N-l,l}$, and write according to this decomposition:
$$
 A = \begin{pmatrix} A_{11} & A_{12} \\ A_{21} & A_{22} \end{pmatrix}.
 $$
 Then $A_{|E_B}$ is precisely scalar iff 
 \begin{align*}
A_{11} + A_{12} B = \lambda \I_l; \qquad A_{21} + A_{22} B = \lambda B
 \end{align*}
 for some $\lambda \in\C$. By substitution and elimination, these are equivalent to the equations:
 \begin{align*}
   A_{11} + A_{12} B =\frac 1 l \tr (A_{11} + A_{12}B)\I_l; \qquad A_{21} + A_{22} B =  B ( A_{11} + A_{12}B )
   \end{align*}
 which are polynomial equations in the entries of $A$ and $B$. Fixing a plane $E_B$ in this chart of $\Gr(l,N)$ amounts to fixing $B$, so that the above equation leaves $N^2 - l^ 2- l(N-l)+1 = N(N-l)+1$ free variables (the $+1$ being the scalar $\lambda$). This shows that we have a local trivialization:
 $$
\rho^{-1}( U ) \cong  \C^{N(N-l)+1} \times U.
$$
This shows that $\cE$ is a (holomorphic) vector bundle over $\Gr(l,N)$ which as a complex manifold has total dimension $\dim_\C \cE = N(N-l)+1 +  l(N-l) = N^2 + 1 -l^2$. Note that $\cE$ is closed as a complex submanifold: for each chart $U \subset \Gr(l,N)$ the inverse image $\rho^{-1}(U)$ is an algebraic, hence closed subset in $M_N(\C) \times U$ so that, for a (finite) open coordinate cover $\mathcal U$ of $\Gr(l,N)$:
$$
M_N(\C) \times \Gr(l,N) \setminus \cE = \bigcup_{U \in \mathcal U} (M_N(\C) \times U) \setminus \cE = \bigcup_{U \in \mathcal U} \left( (M_N(\C) \times U )\setminus  \rho^{-1}(U) \right)
$$
which is a finite union of open subsets, showing that $\cE$ is closed in $M_N(\C) \times \Gr(l,N)$. 
\endproof

\begin{lma}
The set $\cN$ is a closed analytic subset of $M_N(\C)$ which has codimension $\geq l^2-1$. 
  \end{lma}
\proof
As already mentioned, the restriction to $\cE$ of the projection to the first coordinate has image exactly $\cN$; call this map $F: \cE \to \cN$. Since the Grassmannian is compact, this projection is proper (hence maps closed subsets to closed subsets). It now follows from Remmert's proper mapping theorem (see for instance  also \cite[Theorem 1.1]{BN90} or \cite[Theorem 3.2]{Chi89}) that the image $\cN$ is an analytic subset. 

Moreover, dimension under a holomorphic mapping cannot increase ({\em cf.} \cite[Corollary 8.6]{Dem12}) so that $\dim \cN \leq \dim \cE$ and Lemma \ref{lma:dim-R} then implies that $\dim (\cE) = N^2 - l^2 + 1$. Hence, $\text{codim}_\C (\cN) \geq N^2-(N^2-l^2+1) =l^2-1$.
\endproof

\begin{lma}
  The space $U(p,q)\setminus \cN$ is path-connected.
  \label{lma:path-conn-U}
\end{lma}
\proof
First intersect $\cN$ with the open subset $GL_{nm}(\C) \subset M_{nm}(\C)$  and conclude that $\cN^\times:= \cN \cap GL_{nm}(\C)$ is a closed analytic subset of $GL_{nm}(\C)$ with complex codimension $\geq 3$ (since $l=m+1 \geq 2$). 
We will use that $U(p,q)$ is the fixed point subset in $GL_{nm}(\C)$ for the anti-holomorphic involution
$$
\sigma : GL_{nm}(\C) \to GL_{nm}(\C); A \mapsto \Lambda \cdot (A^*)^{-1} \Lambda
$$
in terms of the matrix $\Lambda = \I_p \oplus -\I_q$ introduced before. The set $\cN$ is invariant under this involution as well, since 
$$
\dim \ker (\sigma(A) - \lambda \I) = \dim \ker ( A- \bar\lambda^{-1} \I)
$$
so that by Proposition \ref{prop:m-cyclic-dimker} $m$-cyclicity is $\sigma$-invariant.
Hence $\cN \cap U(p,q) = (\cN^\times)^\sigma$.

We start by decomposing $\cN^\times$ into regular and singular parts, their singular parts, and so on. So $\cN^\times = \sqcup_{j} \cN_j$ where each $\cN_{j+1}^\times: = \text{Sing}(\cN^\times_j)$ and $\cN_0^\times :=\cN^\times$. Write
$$
S_{j,d} := \{ A \in \text{Reg}(\cN_j^\times) \mid \dim_{\C,A} \cN_j = d \}.
$$
The $S_{j,d}$ are either empty or locally closed complex manifolds of $\dim_\C = d \geq n^2m^2-3$, whose union is the closed subset $\cN$.

Since an anti-holomorphic isomorphism preserves regular points, singular points and local dimension, the $S_{j,d}$ are also $\sigma$-invariant. We claim that the fixed-point subsets $S_{j,d}^\sigma$ are smooth real submanifolds of $U(p,q)$ of real codimension $\geq 3$. In fact, we may choose a Riemannian metric on $U(p,q)$ for which $\sigma$ is an isometry (possibly after averaging over $\sigma$). Then, around a fixed point $A \in S_{j,d}^\sigma$ we may decompose
$$
T_A S_{j,d} = ( T_A S_{j,d} )^{d \sigma_A} \oplus i ( T_A S_{j,d} )^{d \sigma_A}
$$
and use the exponential map to obtain a neighborhood of $A$ which is locally isomorphic to $(T_A S_{j,d})^{d \sigma_A}$. We conclude that we may decompose the closed subspace $U(p,q) \cap \cN = (\cN^\times)^\sigma = \sqcup_{j,d} S_{j,d}^\sigma$ with each piece $S_{j,d} ^\sigma$ a real smooth submanifold in $U(p,q)$ of codimension $\geq 3$. 

Now, take $A_0,A_1 \in U(p,q) \setminus (\cN^\times)^\sigma$. Since $U(p,q)$ deformation retracts to $U(p) \times U(q)$ it is path-connected. Hence, there is a path $\gamma:[0,1] \to U(p,q)$ from $A_0= \gamma(0)$ to $A_1 =\gamma(1)$; we rescale to make this path constant near its endpoints. It can then be extended to all of $\R$ as a smooth map $\gamma: \R \to U(p,q)$. On each $S_{j,d}^\sigma$ consider
$$
\Psi_{j,d} : S_{j,d}^\sigma \times \R \to U(p,q) ; \qquad (A,t) \mapsto A \gamma(t)^{-1}. 
$$
The domain has dimension $n^2m^2-3+1 = n^2m^2 -2< \dim_\R U(p,q)$. By Sard's Theorem its image $\Psi_{j,d} ( S_{j,d}^\sigma \times \R )$ in $U(p,q)$ has measure zero, and the same applies to the finite union,
$$
V := \bigcup_{j,d} \Psi_{j,d} ( S_{j,d}^\sigma \times \R ) \subseteq U(p,q). 
$$
Note that for $W \in U(p,q)$ and $t \in \R$ we have that $W \gamma(t) \in S_{j,d}^\sigma$ implies that $ W \in V$. In other words, $W \notin V$ implies that $W\gamma(t) \notin (\cN^\times)^\sigma$ for all $t \in \R$. 

Since $(\cN^\times)^\sigma$ is closed, and $A_0,A_1 \notin (\cN^\times)^\sigma$, we may choose path-connected neighborhoods $U$ of the identity in $U(p,q)$ such that
$$
UA_0 \cup UA_1 \subseteq U(p,q) \setminus (\cN^\times)^\sigma.
$$
The measure-zero set $V$ cannot exhaust $U$, so choose $W \in U \setminus V$, and a path $\eta: [0,1] \to U$ from $1$ to $W$. Then, by the above $W \gamma(t) \notin (\cN^\times)^\sigma$ and the concatenation of paths
$$
t  \mapsto \eta(t) A_0, \qquad t \mapsto W \gamma(t), \qquad t \mapsto \eta(1-t) A_1
$$
lies entirely in $U(p,q) \setminus (\cN^\times)^\sigma$. 

Finally, $U(p,q) \setminus (\cN^\times)^\sigma$ is non-empty because we may take a diagonal matrix with distinct eigenvalues of modulus one: it belongs to $U(p,q)$ and is $m$-cyclic as follows from Proposition \ref{prop:m-cyclic-dimker}.
\endproof


\proof[Proof of Proposition \ref{prop:path-conn-C-block}]
Now consider the subsets of interest for the space $\cC{n,m}_{p,q}$. First of all, Lemma \ref{lma:path-conn-U} shows that
$$
M_{nm,m}(\C) \times U(p,q) \setminus  Z_{p,q}^{nc} =
M_{nm,m}(\C) \times (U(p,q) \setminus  \cN )
$$
is path-connected. For $m$-cyclic $W$, the map $X \mapsto \det (K(X,W))$ is a non-zero polynomial. Hence ({\em e.g.} by \cite[Proposition 2.2.3]{Chi89}) it follows that
$$
\Omega_W := M_{nm,m}(\C) \setminus \{ \det K(\cdot, W)  = 0 \}
$$
is path-connected.

Now, take two elements $(X_0,W_0), (X_1, W_1)$ in $\cC{n,m}_{p,q}$ and choose a path $\{ W_t\}_{t \in [0,1]}$  from $W_0$ to $W_1$ in the path-connected space $U(p,q)\setminus \cN$. For every $t_0$ we choose $X_{t_0} \in \Omega_{W_t}$. Then $\det K(X_{t_0},W_t) \neq 0$ for $t$ in a neighborhood of $t_0$. Compactness of $[0,1]$ yields a finite chain of intervals covering $[0,1]$. At an overlap $t_j$ the two choices $X_j,X_{j+1}$ lie in the same path-connected fiber $\Omega_{W_{t_j}}$ so that they can be joined by a path. Concatenating these paths in the fibers with the pieces of the path along $W_t$ yields a path from $(X_0,W_0)$ to $(X_1, W_1)$.
\endproof

\proof[Proof of Theorem \ref{thm:toepl}]
(1) In Proposition \ref{prop:param-Toepl-block} we showed that the map
$$
\pi^{(m)}: \cC{n,m}_{p,q} \to H(\cT{n},m)_{p,q}
$$
is continuous and surjective, so that path-connectedness of $H(\cT{n},m)_{p,q}$ follows from that of $\cC{n,m}_{p,q}$ (Proposition \ref{prop:path-conn-C-block}).

(2) Since the signature of a matrix is a homotopy invariant, the space of non-singular hermitian block Toeplitz matrices splits into path-connected components labeled by signature. For any possible signature these components are non-empty which can be seen as follows. First of all, there exists a $n \times n$ hermitian invertible Toeplitz matrix for any possible signature. Indeed, take the pair $(X,W) \in \cC{n,1}_{p,q} \subset \C^n \times U(p,q)$ given by 
    $$
    X = \begin{pmatrix} 1 \\ 1  \\ \vdots \\ 1 \end{pmatrix}; \qquad W = \begin{pmatrix} \lambda_1 & 0 & \cdots & 0 \\ 0 & \lambda_2 & \cdots & 0 \\ \vdots & \vdots &\ddots & \vdots \\0 & 0 & \cdots & \lambda_n \end{pmatrix}
      $$
    for all different $\lambda_1, \ldots, \lambda_n$ on the unit circle. In this case, $K(X,W)$ is a Vandermonde matrix
    $$
   K(X,W) = 
 \begin{pmatrix} 1 & \lambda_1 & \cdots & \lambda_1^{n-1} \\
  1 & \lambda_2 &  \cdots & \lambda_2^{n-1} \\
  \vdots & \vdots & & \vdots \\
   1 & \lambda_n  & \cdots & \lambda_n^{n-1}
  \end{pmatrix}
 $$
 with determinant $\prod_{1 \leq i<j\leq n} (\lambda_j - \lambda_i) \neq 0$.
 Consequently, $\pi^{(1)}(X,W)$ is a non-singular hermitian Toeplitz matrix with signature $(p,q)$.

 Finally, if we take $m$ direct sums of different combinations of the above matrices, we obtain non-singular hermitian block Toeplitz matrices in $M_m(\cT{n})$ with signatures taking all possible values.
 \endproof

  \begin{rem}
    For the case of (scalar) Toeplitz matrices, it is possible to derive the result that $H(\cT{n},1)$ has exactly $n+1$ path components by making use of the Cauchy index. The main steps in such a proof are as follows:
\begin{enumerate}
\item Find a suitable generating function $f(z)$ for a non-singular Toeplitz matrix $T$ and show that the signature of $T$ is equal to the Cauchy index of $f(z)$, using \cite{HR10} 
  \item For a given index $k \in \{-n,-n+2,\ldots, n\}$ continuously move the poles and zeros of any $f$ to a standard form, as in \cite{Bro76}, without changing the Cauchy index.
\end{enumerate}
This shows that hermitian Toeplitz matrices of a fixed signature form a path-connected component, which can furthermore be shown to be non-empty.
This was further discussed in \cite{Sch26}. It would be interesting to extend this elegant argument invoking index invariants to the block Toeplitz case; we leave this for future research. 
    \end{rem}

  \section{$K$-invariants of the Fej\'er--Riesz operator system}
  \label{sect:FR}
In \cite{CS20} we also considered the operator system dual of the Toeplitz operator system (this was fully established in \cite{Far21}). The so-called Fej\'er--Riesz operator system is a function system (in the sense of \cite{Arv69}) given by Fourier truncations on the circle, and is defined for $n \geq 2$ by
$$
\FR{n} := \{a \in C(S^1) \mid  a(z) =\sum_{k=-n+1}^{n-1} a_k z^k, z \in S^1  \}
$$
so that $a_k=0$ whenever $|k| \geq n$. The notion of positivity is inherited from $C(S^1)$: $a \geq 0$ iff the function $z \mapsto \sum_{k=-n+1}^{n-1} a_k z^k$ is positive on $S^1 \subseteq \C$.
The same applies to the matrix level, where $M_m(C(S^1)_{(n)})$ is identified with the space of $m \times m$ matrix-valued functions in $z \in S^1$ with the same support constraint in Fourier, so that any $a \in M_m(C(S^1)_{(n)})$ corresponds to
$$ 
a (z) = \sum_{k=-n+1}^{n-1} a_k z^k; \qquad (a_k \in M_m(\C)).
$$
The $C^*$-algebra $C(S^1)$ is also the $C^*$-envelope, so non-degeneracy of $a$ corresponds to invertibility of $a(z)$ for all $z$, or, equivalently, to the non-vanishing of $\det (a(z))$ for all $z \in S^1$.

As is well-known \cite[Chapter 6]{BS90}, to any sequence $(a_k)_{k \in \Z}$ of $m \times m$ matrices we may associate a block Toeplitz operator:
\begin{equation}
  \label{eq:T-a}
T(a) := \begin{pmatrix} a_0 & a_{-1} & a_{-2} &\cdots \\
  a_1 & a_0 & a_{-1}  & \ddots  \\
  a_2 & a_1 & a_0 & \ddots & \\
  \vdots & \ddots & \ddots & \ddots\end{pmatrix}.
  \end{equation}
  Moreover,  for $a \in C(S^1, M_m(\C))$ the operator $T(a)$ is bounded on $\ell^2(\N,\C^m)$.
  We also record two more classical results on block Toeplitz operators, referring to \cite{BS90} (Theorem 6.5 and 6.6 specifically).
\begin{thm}[Gohberg]
Let $a \in  C(S^1, M_m(\C))$. Then $T(a)$ is Fredholm on $\ell^2(\N,\C^m)$ if and only if $\det(a(z))$ has no zeros on $S^1$.
\end{thm}
The next result concerns Wiener--Hopf factorization for block Toeplitz operators.

\begin{thm}[Gohberg--Krein]
  Let $(a_k)_{k \in\Z}  \in  \ell^1(\Z,M_m(\C))$ and suppose that $\det a(z)$ does not vanish on $S^1$. Then there exist indices
  $$
\kappa_1 \leq \kappa_2 \leq \cdots \leq \kappa_m
  $$
and invertible matrix-valued functions $a_{+},a_- \in  \ell^1(\Z,M_m(\C))$ with support in $[0,\infty)$ and $(-\infty,0]$, respectively, such that
$$
a(z) = a_-(z) \text{diag}(z^{\kappa_1}, \ldots, z^{\kappa_m}) a_+(z); \qquad (z \in S^1).
$$
Moreover,
$$
\dim \ker T(a) = \sum_{\kappa_j<0} | \kappa_j| ;\qquad \dim \coker T(a) = \sum_{\kappa_j>0} \kappa_j
$$
  \end{thm}

\subsection{Path components of non-singular hermitian matrix-valued functions on $S^1$}

We now consider the $K_0$-invariants $\V_0(\FR{n},m)$ for the Fej\'er--Riesz operator system. From the above it follows that we should consider the number of path-components of the space of invertible hermitian matrix-valued continuous functions on $S^1$ with Fourier support in $(-n,n)$. Note that the space of such hermitian forms can be described as
$$
H(\FR{n},m) := \{ a \in M_m( \FR{n} ) \mid a=a^*, \det a(z) \neq 0 \text{ for all }z \in S^1 \}.
$$ 
  Note that to any $a \in H(\FR{n},m)$ and any $z \in S^1$ we can assign a signature to the matrix $a(z)$; since the signature is a continuous map from $S^1$ to $\Z$, it follows that it is a constant on $S^1$. Hence it makes sense to talk about the signature of $a \in H(\FR{n},m)$ and we denote by $H(\FR{n},m)_{p,q}$ those elements in $H(\FR{n},m)$ with signature $(p,q)$ where $p+q = m$.

  \begin{thm}
    \label{thm:FR}
  \begin{enumerate}
\item     The space $H(\FR{n},m)_{p,q}$ of invertible hermitian matrix-valued functions on $S^1$ with signature $(p,q)$ is path-connected.
    \item The space $H(\FR{n},m)$ of invertible hermitian matrix-valued functions on $S^1$ has $m+1$ path-connected components, which are labeled by matrix signatures. Consequently, $\V_0(\FR{n},m) \cong \{ -m, -m+2, \ldots, m\}$.   
    \end{enumerate}
    \end{thm}
As before, we conclude that the inclusion map $\imath_\env^{\FR{n}} : \FR{n} \to C^*_\env(\FR{n}) \cong C(S^1)$ induces isomorphisms $\V_0(\FR{n},m) \cong \V_0(C(S^1), m)$ for any $m \geq 1$, intertwining the connecting maps $\imath_{m m'}$ from Equation \eqref{eq:dir-syst} for $E= \FR{n}$ with those of $E = C(S^1)$. It follows that  $K_0(\FR{n}) \cong K_0 (C(S^1) ) \cong \Z$.

Before formulating the proof of the above Theorem, let us start with some preliminary results. 
\begin{lma}
  \label{lma:FR-aux1}
  For any $a \in H(\FR{n},m)$ there exists $\epsilon_0>0$ such that for all $0< \epsilon < \epsilon_0$ we have
  \begin{enumerate}
\item $a_t = a+t \epsilon \I_m \in  H(\FR{n},m)$ for all $t\in [0,1]$.
\item The block Toeplitz operator $T(a+ \epsilon \I_m)$ is a self-adjoint Fredholm operator which is invertible.
  \end{enumerate}
\end{lma}
\proof
(1) Since $a$ is invertible, $\| a^{-1} \| \leq C$ for some $C>0$. But then $a(z) + s \I_m$ is invertible provided $|s| < 1/C$; take $\epsilon_0 = 1/C$.

(2)
It follows from the Gohberg Theorem that $T(a)$ is Fredholm. Since hermiticity of $a$ amounts to $a_k^* = a_{-k}$ it is clear from \eqref{eq:T-a} that $T(a)$ is self-adjoint. Since self-adjoint Fredholm operators can only have $0$ in its spectrum $\sigma(T(a))$ as an eigenvalue of finite multiplicity, we may perturb $T(a)$ to an invertible Fredholm operator $T(a+\epsilon \I_m) = T(a) + \epsilon 1$ provided $0< \epsilon<\epsilon_0$ and we shrink $\epsilon_0$ below $\text{dist}(0,\sigma(T(a)) \setminus \{0\} )$.
\endproof

If we apply Wiener--Hopf factorization to an element $a \in H(\FR{n},m)$ with $T(a)$ invertible, we find 
$$
a(z) = a_-(z) a_+(z); \qquad (z \in S^1),
$$
where $a_+$ and $a_-$ are invertible functions which are holomorphic on $\overline \D$ and $\C \setminus \D$, respectively. We now have to establish a form of this factorization for which these functions are actually polynomials (in $z$ and $1/z$, respectively).  

Recall the involution on holomorphic functions: $a^*(z) = a(1/\overline z)^*$, or,  if $a(z) = \sum_j a_j z^j $ then
$$
 a^*(z) = \sum_j a_j ^*z^{-j}.
$$

 \begin{lma}
   \label{lma:FR-aux2}
  Let $a \in H(\FR{n},m)_{p,q}$ be such that $T(a)$ is invertible. Then there exists $b \in M_m(C(S^1))$ so that the function
  $$
b(z) = \sum_{j=0}^{n-1} b_j z^j
  $$
has $\det b(z) \neq 0$ for all $|z| \leq 1$, and for which 
$$
a(z) = b(z)^* \cdot S \cdot b(z)
$$
where $S$ is a constant hermitian invertible matrix with signature $(p,q)$.  
\end{lma}
\proof
Since $a = a_- a_+$ and $a=a^*$ it follows that $a_- a_+ = a_+^* a_-^*$. Let us define
$$
S:= (a_+^*)^{-1} a_- = a_-^* a_+^{-1}.
$$
Then $(a_+^*)^{-1} a_-$ has Fourier support in $(-\infty,0]$ while $a_-^* a_+^{-1}$ has support in $[0,\infty)$. This implies that $S$ is constant. Moreover, from the equation it also follows that $S$ is invertible. 

Now write $b = a_+$ so that $a_- = b^*\cdot S$ and also $a= b^*\cdot S \cdot b$. Since $a(z)$ is hermitian on the circle, it follows that $b^*(z) S b(z)$ is hermitian on the circle, hence $S$ is self-adjoint. 

The degree of $b$ can be computed by multiplying $a = b^* \cdot S \cdot b$ on the left by $(b^{-1})^*$:
$$
(b^{-1})^*\cdot  a = S \cdot b.
$$
Since the support of $(b^{-1})^*$ lies in $(-\infty,0]$ while the support of $a$ is contained in $(-n,n)$ we find that their product $(b^{-1})^* \cdot a$ is supported in $(-\infty,n)$. The right-hand side $S\cdot b$ is supported in $[0,\infty)$; hence $b$ is supported in $[0,n)$.

 For the zeros of $\det b(z)$ we already know that $\det b(z) \neq 0$ for all $z \in \D$. Moreover, $\det b(z) \neq 0$ on $S^1$ because $a(z) = b^*(z) S b(z)$ is pointwise invertible. It thus follows that $\det b(z) \neq 0$ on $\overline \D$. 
      
 Finally, that the signature of $S$ is equal to that of $a(z)$ follows by Sylvester's Law.
\endproof

\begin{prop}
  \label{prop:FR-aux}
  Let $a \in H(\FR{n},m)$ be such that $T(a)$ is invertible. Then $a$ can be joined inside $H(\FR{n},m)$ to a constant hermitian invertible matrix.
\end{prop}
\proof
Write $a = b^* \cdot S \cdot b$, set $b_t(z) = b(tz)$ and write $a_t(z) = b_t^*(z) \cdot S \cdot b_t(z)$. Then we have 
$$
a_t(z)  =  \sum_{j,k=0}^{n-1} t^{j+k} b_j^* \cdot S \cdot b_k z^{k-j} .
$$
As a matrix-valued polynomial this is continuous in $t$ and since $|tz| \leq 1$ and $\det b(z) \neq 0$ on $\overline \D$ we find that $b(tz)$ is invertible. Hence $a_t(z)$ is invertible and with constant signature equal to that of $S$. The endpoints are $a_1(z) = a(z)$ and $a_0(z) = b_0^* S b_0$, as desired. 
\endproof

\proof[Proof of Theorem \ref{thm:FR}]
(1) We claim that any $a \in H(\FR{n},m)_{p,q}$ can be joined by a continuous path to the constant matrix $S = \I_p \oplus - \I_q$. First, by Lemma \ref{lma:FR-aux1} there exists a path in $H(\FR{n},m)$ of constant signature from $a$ to $a'$ for which the block Toeplitz operator $T(a')$ is self-adjoint Fredholm and invertible. Then Proposition \ref{prop:FR-aux} applies and shows that $a'$ can be joined inside $H(\FR{n},m)_{p,q}$ to a constant invertible hermitian $m \times m$ matrix of signature $p,q$. But we already know from Example \ref{ex:V-C} that the space of the latter such matrices is path-connected.

(2) The above shows that $\V_0(C(S^1)_{(n)},m) \cong \V_0(\C,m)$ so that the result follows from \eqref{eq:V-C-n} in Example \ref{ex:V-C}.
  \endproof

\subsection{Path components of non-singular matrix-valued functions on $S^1$}
Continuing with the Fej\'er--Riesz operator system, we now consider the $K_1$-invariant $\V_1(\FR{n},m)$. Note that since the $C^*$-envelope of this system is infinite-dimensional, Proposition \ref{prop:K1} does not apply. 

The $K_1$-invariants are based on the space of all invertible elements in $M_m(\FR{n})$, which is
$$
G(\FR{n},m ):= \left\{ a \in M_m(\FR{n}) \mid \det a(z) \neq 0 \text{ for all } z \in \S^1 \right\}.
$$
Note that to an invertible $a(z)$ ($z \in S^1$), we may associate a determinant winding number to it, denoted by $\textup{wind}_{S^1}(\det a(z))$. In fact, it will turn out that this completely determines the path components of $G(\FR{n},m )$.

\begin{thm}
  \label{thm:FR-K1}
  The determinant winding number induces an isomorphism
\begin{align*}
  \V_1 ( \FR{n},m) &\to \{ -m(n-1), -m(n-1)+1, \ldots, m(n-1) \};\\
    [a] &\mapsto \textup{wind}_{S^1}(\det a(z)).
  \end{align*}
\end{thm}
This should be contrasted to $\V_1(C(S^1),m)$ where all winding numbers may occur; in fact, $\V_1(C(S^1),m) \cong \Z$.

Rather than the matrix-valued Laurent polynomials $a(z)$ we formulate our proof in terms of matrix polynomials $p(z)$ of degree $d := 2n-2$, defined by
$$
p(z) = z^{n-1} a(z) = \sum_{j=0}^d p_j z^j; \qquad p_j = a_{j-n+1}\in M_m(\C).
$$
Indeed, $p(z)$ is invertible on $S^1$ iff $a(z)$ is. Denote by $r(p)$ the number of zeros of $P$ on the open disc $\D$, counting multiplicities. By the argument principle,
$$
\textup{wind}_{S^1}(\det a(z)) = r(p) - (n-1) m.
$$
Since winding number is constant along paths, we need to prove that the following sets are path connected:
$$
\mathcal P_r := \{ p(z) \mid \deg p \leq d, \det p(z) \neq 0 \text{ for all } z \in S^1, r(p) = r\}.
$$

\begin{lma}
  \label{lma:FR-K1}
Any $p \in \mathcal P_r$ can be connected to a matrix polynomial $q \in \mathcal P_r$ which is monic, {\em i.e.} $q_{d} = \I_m$ and which has simple determinant roots. 
  \end{lma}
\proof
Let $p \in \mathcal P_r$ and choose $\epsilon < \min \sigma(p(z))$ for all $z \in S^1$ (smaller than all singular values of $p(z)$ on $S^1$). Consider the open neighborhood $U_\epsilon$ defined by $\sum_j \| p_j -q_j \| < \epsilon$. Then the whole path $p_t = (1-t) p + t q$ lies in $\mathcal P_r$ since
$$
\min \sigma( p_t(z)) \geq \epsilon - t \sum_j \| p_j -q_j \| >0.
$$
Hence $\det p_t(z) \neq 0$ for all $z \in S^1$ and for all $t \in [0,1]$, so that its winding number is well-defined. But then the latter is constant along the path $p_t$, so that $r(p_t) = r$. We have thus shown that $U_\epsilon$ is an open neighborhood of $p$ which is contained in $\mathcal P_r$.

The conditions $\det q_d \neq 0$ and $\text{Disc}_z(\det q(z)) \neq 0$ (for the discriminant of $\det q(z)$) specify matrices with invertible top coefficient matrix and with simple determinant roots. Since $(\det q_d )\cdot \text{Disc}_z(\det q(z))$ is a non-zero polynomial in the matrix coefficients of $q_0,\ldots, q_d$, it cannot vanish on the open subset $U_\epsilon$. Hence, there exists $q \in U_\epsilon  \subseteq \mathcal P_r$ with invertible top coefficient matrix and with simple determinant roots. Taking a path $g_t$ in $GL_m(\C)$ from $\I_m$ to $q_{d}^{-1}$ we find that $g_t q(z)$ is a continuous path staying within $\mathcal P_r$ from $q$ to a monic polynomial and with simple determinant roots. 
\endproof

Instead of $\mathcal P_r$ we are thus led to consider the space $\M_r$ of monic polynomials of degree $d$, invertible on $S^1$ and with simple determinant roots. We will parametrize these by the following space
$$
X_r := \left\{ ( \lambda_1, \ldots, \lambda_{md} ; v_1, \ldots, v_{md} ) \in \D^r \times (\C \setminus \overline \D)^{md-r} \times (\C^m)^{md} \mid \det V \neq 0, \lambda_i \neq \lambda_j \right\}
$$
where the matrix $V\in M_d(\C)$ is defined by
$$
V = \begin{pmatrix} v_1 &\cdots &v_{md} \\
  \lambda_1 v_1 & \cdots & \lambda_{md} v_{md}\\
  \vdots & & \vdots \\
  \lambda_1^{d-1} v_1 & \cdots & \lambda_{md}^{d-1} v_{md}\\ 
  \end{pmatrix}.
$$
\begin{prop}
  \label{prop:FR-K1}
  \begin{enumerate}
  \item The following defines a continuous map $X_r \to \M_r$,
    $$
    ( \lambda_1, \ldots, \lambda_{md} ; v_1, \ldots, v_{md} ) \mapsto p(z) = z^d \I_m + \sum_{j=0}^{d-1} p_j z^j
    $$
    where $p_0, \ldots, p_{d-1}$ are determined by
    \begin{equation}
      \label{eq:p-j}
    \begin{pmatrix} p_0 &\cdots & p_{d-1} \end{pmatrix}
    = - \begin{pmatrix} \lambda_1^d v_1 & \cdots \lambda_{md}^d v_{md} \end{pmatrix} \cdot V^{-1}. 
    \end{equation}
    \item This map $X_r \to \M_r$ is surjective. 
\item The space $X_r$ is path-connected. 
  \end{enumerate}
 \end{prop}
\proof
(1) We may multiply \eqref{eq:p-j} by $V$ on the right-hand side to obtain
$$
 \begin{pmatrix} p_0 &\cdots & p_{d-1} \end{pmatrix} \cdot \begin{pmatrix} v_1 &\cdots &v_{md} \\
  \lambda_1 v_1 & \cdots & \lambda_{md} v_{md}\\
  \vdots & & \vdots \\
  \lambda_1^{d-1} v_1 & \cdots & \lambda_{md}^{d-1} v_{md}\\ 
  \end{pmatrix} +  \begin{pmatrix} \lambda_1^d v_1 & \cdots \lambda_{md}^d v_{md} \end{pmatrix}  = 0 .
 $$
 Equivalently, we have
 $$
\begin{pmatrix} p(\lambda_1) v_1 & \cdots &  p(\lambda_{md}) v_{md} \end{pmatrix}  =0 .
$$
This means that each $\lambda_i$ is a root of $\det p(z)$, so that
$$
\det p(z) = \prod_{i=1}^{md} (z- \lambda_i) .
$$
Hence, $p \in \M_r$ as desired.

(2) Take any $p \in \M_r$ and label roots by $\lambda_1, \ldots, \lambda_r \in \D$, $\lambda_{r+1} ,\ldots, \lambda_{md} \in \C \setminus \overline \D$ and choose $v_i \in \ker P(\lambda_i)$. Then the vectors $w_i := \begin{pmatrix} v_i & \lambda_i v_i &\cdots & \lambda_i^{d-1} v_i \end{pmatrix}^t$  are eigenvectors of the companion matrix,
$$
C = \begin{pmatrix} 0 & 1 & 0 & \cdots & \\
 0  & 0 & 1 & \\
 \vdots & & \ddots &\ddots &\\
 0 & \cdots &0& 0 &1\\
   -p_0 & -p_1 & -p_2 & \cdots & p_{d-1}\end{pmatrix}
$$
with eigenvalues $\lambda_i$, respectively, as one can readily check. As these are all distint, the vectors $w_i$ are linearly independent. The matrix $V$ defined as above thus has $\det V \neq 0$.

(3) The space $X_r$ is defined as the complement in $\D^r \times (\C \setminus \overline \D)^{md-r} \times \C^{md}$ of the zero locus of the polynomial
$$
(\det V) \cdot  \prod_{i<j} ( \lambda_j -\lambda_i).
$$
It thus follows from {\em e.g.} \cite[Proposition 2.2.3]{Chi89} that $X_r$ is path-connected. 
\endproof

\proof[Proof of Theorem \ref{thm:FR-K1}]
By Lemma \ref{lma:FR-K1} it follows that any $p \in \mathcal P_r$ can be connected to a matrix polynomial in $\M_r \subseteq \mathcal P_r$. But by Proposition \ref{prop:FR-K1} the latter is a continuous image of the path-connected set $X_r$, so it is itself path-connected. This means that $\mathcal P_r$ is path-connected.

Finally, all integers from $-m(n-1)$ to $m(n-1)$ occur as determinant winding numbers, since
$$
a(z) = \textup{diag} (z^{a_1}, \ldots, z^{a_m} ) \in G(\FR{n},m); \qquad -n+1 \leq a_j \leq n-1
$$
has winding number $\sum a_j$.
\endproof

  \section{$K$-invariants of graph operator systems}
  \label{sect:tol}
Let $X$ be a finite set and let $\mathcal{R}\subseteq X\times X$ be a reflexive and symmetric relation. Such a relation is called a \emph{tolerance relation}. A finite set with a tolerance relation $\mathcal{R}$ can be interpreted as a graph with vertex set $X$ and edge set $\cR$. We will loosely interchange $\cR$ with this graph. A tolerance relation $\mathcal{R}$ on $X$ gives rise to a finite-dimensional operator system $E(\mathcal{R})\subseteq M_n(\C)$ as follows. We may assume without loss of generality that $X=\{1,\dots,n\}$ and then we define
\begin{equation}
    E(\mathcal{R}) = \{(x_{ij})\in M_n(\C)\mid x_{ij}=0\text{ if }(i,j)\notin \mathcal{R}\}.
\end{equation}

These operator systems are the operator systems associated to tolerance relations on finite sets that are considered in \cite{OP15,CS21}. The $C^*$-envelopes and propagation numbers of these operator systems are known. Indeed, a slight generalization of Proposition 3.7 in \cite{CS21} states that
\begin{equation}
  \label{eq:Cstar-ER}
    C_\env^*(E(\R)) = \bigoplus_{K\in\mathcal{K}} M_{|K|}(\C),
\end{equation}
where $\mathcal{K}$ is the set of equivalence classes in $X$ of the equivalence relation generated by $\mathcal{R}$ so that $|\mathcal {K}| = c$, the number of connected components of the corresponding graph. This is also in agreement with the description of the $C^*$-envelope of a graph operator system obtained in \cite[Theorem 3.2]{OP15}.

\subsection{Path components of hermitian forms in graph operator systems}

\begin{thm}
  \label{thm:tol}
  \begin{enumerate}
    \item 
      Let $\cR$ be connected. The space of non-singular hermitian matrices in $M_m( E(\cR))$
with signature $(p,q)$, denoted by $H(E(\cR),m)_{p,q}$, is path-connected.
\item The space $H(E(\cR),m)$ has $\prod_{K \in \cK} (|K| m+1)$ path-connected components, which are labeled by the tuple of matrix signatures in each summand. Consequently,
  $$
  \V_0(E(\cR),m) \cong \prod_{K \in \cK} \{-|K| m, -|K| m+2, \ldots, |K| m\}.
  $$
      \end{enumerate}
\end{thm}
Thus, the inclusion map $\imath_\env^{E(\cR)} : E(\cR) \to C^*_\env(E(\cR)) \cong \oplus_{K \in \cK} M_{|K|}(\C)$ induces isomorphisms $\V_0(E(\cR),m) \cong \V_0( \oplus_{K \in \cK} M_{|K|}(\C) , m)$ for any $m \geq 1$, intertwining the connecting maps $\imath_{m m'}$ from Equation \eqref{eq:dir-syst} for $E= E(\cR)$ with those of $E = \oplus_{K \in \cK} M_{|K|}(\C)$. Combining this with Proposition \ref{prop:K0-properties}(3) it follows that  $K_0(E(\cR)) \cong K_0 (\oplus_{K \in \cK} M_{|K|}(\C)) \cong \Z^{|\cK|}$. 

The proof relies on the Schur complement of a hermitian form, allowing for an inductive argument on the number of elements in $X$. We start with some preparatory remarks that simplify the proof. 

First of all, the semisimple structure of the $C^*$-envelope is inherited by the operator system, in that $E(\cR)$ decomposes into operator subsystems corresponding to each summand. As these summands corresponds to the connected components of the graph, we may thus reduce the proof to the case that $\cR$ is a connected graph, so that $|\cK| = 1$. Moreover, to $\cR$ we may induce a tolerance relation $\tilde \cR$ on the set $\{1,\ldots, n\}  \times \{ 1,\ldots, m\}$ defined by
\begin{equation}
  \label{eq:rel-tildeR}
\tilde \cR := \{ ((i,\alpha), (j,\beta)  ) \mid (i,j) \in \cR \}.
\end{equation}
In other words, $\tilde \cR$ is the edge set of the Kronecker product of the graph $(X,\cR)$ and the full graph on $m$ vertices. 
Clearly, if $\cR$ is connected, then so is $\tilde \cR$. Moreover, since
$$
M_m(E(\cR)) \cong E(\tilde \cR)
$$
this reduces our study of $H(E(\cR),m)$ for connected $\cR$ to the scalar case, $m=1$, {\em i.e.} invertible hermitian elements in the operator system itself. Let us introduce the short-hand notation $E(\cR)_h^\times:= H(E(\cR),1)$ for this space. 

The space $\C^n$ on which $E(\cR)$ acts has a canonical basis $(e_i)_{i \in X}$; in the following we will also make use of (orthogonal complements with respect to) the standard inner product. 
\begin{lma}
  \label{lma:tol-aux1}
  Let $a \in E(\cR)_h^\times$ and let  $i \in X$. Consider $W = e_i^\perp$ with corresponding orthogonal projection $P_W$. Then there exists $a' \in E(\cR)_h^\times$ such that $a'$ is homotopic in $E(\cR)_h^\times$ to $a$ and the compression $P_W a' P_W$ is invertible on $W$.
  \end{lma}
\proof
Let $A = (P_W a P_W )_{|W}$ and let $\delta$ be its gap, {\em i.e.} its smallest non-zero singular value. Choose $0< \epsilon < \text{min} \{ \delta, \|a^{-1} \|^{-1} \}$. For any $t \in [0,1]$ write $a_t = a+ t \epsilon P_W$; this is a continuous path of self-adjoint elements in $E(\cR)$. Moreover, since
$$
\| a_t -a \| \leq \epsilon < \| a^{-1} \|^{-1}
$$
it follows that $a_t$ is invertible for all $t$. Finally, if we write $a' = a_1$ then $A' = (P_W a' P_W)_{|W} = A + \epsilon 1_W$ is invertible since $\epsilon < \delta$.  
\endproof

\begin{lma}
  \label{lma:tol-aux2}
  Let $a \in E(\cR)_h^\times$ and let $i \in X$. With respect to the decomposition $\C^{n} = e_i^\perp \oplus \C \cdot e_i$  write
  $$
a = \begin{pmatrix} A & v \\ v^* & d \end{pmatrix}; \qquad ( A \in M_{n-1}(\C), v \in \C^{n-1}, d \in \R).
$$
Assuming that $A$ is invertible, write $s= d- v^* A^{-1} v$ for the $1 \times 1$ Schur complement of $A$. Then $s$ is non-zero and $a$ is homotopic in $E(\cR)_h^\times$ to $A \oplus s$ 
  \end{lma}
\proof
Since
$$
a = \begin{pmatrix} A & v \\ v^* & d \end{pmatrix} =\begin{pmatrix} 1 & 0 \\ v^* A^{-1}  & 1 \end{pmatrix}\begin{pmatrix} A &  0 \\ 0 & s \end{pmatrix} \begin{pmatrix} 1 & A^{-1} v \\ 0 & 1 \end{pmatrix},
$$
the determinant of the Schur complement satisfies $\det a = \det A \cdot s$, whence $s \neq 0$. Moreover, the path 
$$
a_t = \begin{pmatrix} A & t v \\ t v^* & s+ t^2 v^* A^{-1} v \end{pmatrix} ; \qquad (t \in [0,1])
$$
has $\det a_t = \det A \cdot s \neq 0$ since for all $t$ the Schur complement of $a_t$ is readily found to be $s$. We conclude that $a_t \in E(\cR)_h^\times$ while $a_1 = a$ and $a_0 = A \oplus s$.
\endproof

\begin{prop}
  \label{prop:tol}
  \begin{enumerate}
    \item 
  Let $a \in E(\cR)_h^\times$ have signature $(p,q)$. Then there is a continuous path connecting $a$ to a diagonal matrix $\Delta$ with only $\pm 1$ on the diagonal (with signature $(p,q)$).
\item If $\cR$ is connected and $\Delta,\Delta'$ are two $n \times n$ signature matrices (only $\pm 1$'s on the diagonal) with the same signature, then $\Delta$ and $\Delta'$ can be joined by a path in $E(\cR)_h^\times$.
  \end{enumerate}
\end{prop}
\proof
We first show that $a$ is homotopic to a diagonal matrix $\Delta$ with only $\pm 1$ on the diagonal. Induction on the number of vertices starts with $n=1$ for which $\cR= \{ (1,1) \}$. Hence $E(\cR) =  \C$ so that the result follows from Example \ref{ex:V-C}.

Let us then assume $n \geq 2$, and take $i \in X$, setting $W = e_i^\perp$ as before. By Lemma \ref{lma:tol-aux1} it follows that $a$ is homotopic to $a'$ with compression $A' =(P_W a' P_W)_{|W}$ invertible on $W$. Applying Lemma \ref{lma:tol-aux2} to $a'$ then shows that $a'$ is homotopic to $A' \oplus s$ with $s$ a non-zero real number. 

Define $\cR' = \cR_{|\{ 1,\ldots \widehat{i}, \ldots, n\}}$; it is the graph that omits the vertex $i$ and all edges connected to it. Then $A' \in E(\cR')_h^\times$ for which the induction hypothesis gives a diagonal matrix with signature equal to that of $A'$.

Since signature is homotopy invariant, it follows that the resulting matrix can be joined within $E(\cR)_h^\times$ to a matrix $\Delta$ with $p$ times $+1$ and $q$ times $-1$ on the diagonal.

(2) For any two signature matrices $\Delta, \Delta'$ with the same signature, there exists a permutation $\sigma \in S_n$ such that
$$
\Delta'_{ii} = \Delta_{\sigma(i) \sigma(i)}. 
$$
Since $\cR$ is connected, any such permutation can be written as a product of transpositions along the edges of the corresponding graph, {\em i.e.} as a product of exchanges $i \leftrightarrow j$ with $(i,j) \in \cR$. So it is sufficient to show that $\Delta, \Delta'$ can be joined by a path for the case that $\Delta'$ is obtained from $\Delta$ by interchanging the $i$'th and $j$'th diagonal entries. Let $L = \text{span}_\C \{ e_i ,e_j\}\subseteq \C^n$ so that $\Delta_{|L^\perp} = \Delta'_{|L^\perp}$. If $\Delta_{|L} = \Delta'_{|L}$ there is nothing to prove so let us assume that
\begin{align*}
  \Delta_{|L} = \begin{pmatrix} 1 & 0 \\ 0 & -1\end{pmatrix} \qquad 
    \Delta'_{|L} = \begin{pmatrix} -1 & 0 \\ 0 & 1\end{pmatrix} 
\end{align*}
(or the other way around, which amounts to exchanging $\Delta$ and $\Delta'$). We introduce 
$$
C(\theta) = \begin{pmatrix} \cos \theta & \sin \theta \\ \sin \theta & -\cos \theta \end{pmatrix} \oplus \Delta_{|L^\perp} \in E(\cR)_h^\times; \qquad (\theta \in [0,\pi]). 
$$
Upon varying $\theta$ from 0 to $\pi$ we find that $C(0) = \Delta$ while $C(\pi) = \Delta'$ thus providing the sought-for path.
\endproof

\begin{corl}
  \label{corl:tol}
  Assume that $\cR$ is connected and let $a \in E(\cR)_h^\times$ have signature $(p,q)$. Then there is a continuous path connecting $a$ to the diagonal $n \times n$ matrix $\Lambda = \I_p \oplus -\I_q$ (in the canonical basis of $\C^{n}$). 
\end{corl}

\proof[Proof of Theorem \ref{thm:tol}]
(1) Note that $H(E(\cR), m) = H( E(\tilde \cR),1) \equiv E(\tilde \cR)_h^\times$ where $\tilde \cR$ is the connected relation defined in \eqref{eq:rel-tildeR}. We may thus apply Corollary \ref{corl:tol} to the relation $\tilde \cR$, establishing the claimed result. 

(2) According to the decomposition in \eqref{eq:Cstar-ER} the operator system $E(\cR)$ allows a direct sum decomposition:
$$
E(\cR) = \bigoplus_{K \in \cK} E(\cR_K) ; \qquad E(\cR_K) \subseteq M_{|K|}(\C),
$$
with $\cR_K$ the restriction of $\cR$ to the equivalence class $K \subseteq X$. From Proposition \ref{prop:K0-properties}(2) it then follows that $\V_0(E(\cR),m) \cong \prod_{K \in \cK} \V_0( E(\cR_K),m)$. The stated result then follows from (1) of the Theorem.
\endproof


\newcommand{\noopsort}[1]{}

\end{document}